\pdfoutput=1
\RequirePackage{ifpdf}
\ifpdf 
\documentclass[pdftex]{sigma}
\else
\documentclass{sigma}
\fi

\numberwithin{equation}{section}

\newtheorem{Theorem}{Theorem}[section]
\newtheorem*{Theorem*}{Theorem}

\newtheorem{Proposition}[Theorem]{Proposition}
 { \theoremstyle{definition}

\newtheorem{Remarks}[Theorem]{Remarks} }

\usepackage{tikz-cd}
\DeclareMathOperator\C{\mathbb C}
\DeclareMathOperator\Z{\mathbb Z}

\newcommand{\dontprint}[1]\relax

\newcommand{\hra}{\hookrightarrow}
\newcommand{\we}{\wedge}

\renewcommand{\P}{{\mathbb P}}
\newcommand{\A}{{\mathbb A}}

\newcommand{\ot}{\otimes}
\newcommand{\Hom}{\operatorname{Hom}}
\newcommand{\Ext}{\operatorname{Ext}}

\newcommand{\OO}{{\mathcal O}}

\newcommand{\sub}{\subset}

\newcommand{\om}{\omega}
\newcommand{\la}{\lambda}
\renewcommand{\a}{\alpha}

\newcommand{\id}{\operatorname{id}}

\newcommand{\G}{{\mathbb G}}

\newcommand{\lan}{\langle}
\newcommand{\ran}{\rangle}

\renewcommand{\k}{{\mathbf{k}}}

\begin{document}
\allowdisplaybreaks

\newcommand{\arXivNumber}{2301.13417}

\renewcommand{\PaperNumber}{059}

\FirstPageHeading

\ShortArticleName{Ten Compatible Poisson Brackets on $\mathbb P^5$}

\ArticleName{Ten Compatible Poisson Brackets on $\boldsymbol{\mathbb P^5}$}

\Author{Ville NORDSTROM~$^{\rm a}$ and Alexander POLISHCHUK~$^{\rm ab}$}

\AuthorNameForHeading{V.~Nordstrom and A.~Polishchuk}

\Address{$^{\rm a)}$~Department of Mathematics, University of Oregon, Eugene, OR 97403, USA}
\EmailD{\href{mailto:villen@uoregon.edu}{villen@uoregon.edu}, \href{mailto:apolish@uoregon.edu}{apolish@uoregon.edu}}

\Address{$^{\rm b)}$~National Research University Higher School of Economics, Moscow, Russia}

\ArticleDates{Received February 18, 2023, in final form August 03, 2023; Published online August 13, 2023}

\Abstract{We give explicit formulas for ten compatible Poisson brackets on $\P^5$ found in arXiv:2007.12351.}


\Classification{53D17; 18G70; 14H52}

\section{Introduction}

The goal of this paper is to present explicit formulas for certain algebraic Poisson brackets on~$\P^5$.

Recall that two Poisson brackets $\{\cdot,\cdot\}_1$, $\{\cdot,\cdot\}_2$ are called {\it compatible} if any linear combination
$\{\cdot,\cdot\}_1+\la\cdot \{\cdot,\cdot\}_2$ is still a Poisson bracket (i.e., satisfies the Jacobi identity).
Pairs of compatible Poisson brackets play an important role in the theory of integrable systems.

With every normal elliptic curve $C$ in $\P^n$ one can associate naturally a Poisson bracket on $\P^n$, called a {\it Feigin--Odesskii bracket
of type $q_{n+1,1}$}. The corresponding quadratic Poisson brackets on $\A^{n+1}$ arise as quasi-classical limits of Feigin--Odesskii elliptic algebras.
On the other hand, they can be constructed using the geometry of vector bundles on $C$ (see~\cite{FO95, P98}).

It was discovered by Odesskii--Wolf~\cite{OW} that for every $n$ there exists a family of $9$ linearly independent mutually compatible Poisson brackets on $\P^n$,
such that their generic linear combinations are Feigin--Odesskii brackets of type $q_{n+1,1}$. In~\cite{HP}, this construction was explained and extended in terms of anticanonical
line bundles on del Pezzo surfaces. It was observed in~\cite[Example~4.6]{HP} that in this framework one also obtains $10$ linearly independent mutually compatible
Poisson brackets on $\P^5$. In this paper, we will produce explicit formulas for these $10$ brackets (see Theorem \ref{main-thm}).

\section[Homological perturbation for P\^{}n]{Homological perturbation for $\boldsymbol{\P^n}$}

\subsection{Formula for the homotopy}

Let
\[
H=\bigoplus_{p\geq 0,\, q\in \Z}H^p(\P^n,\OO(q))
\]
be the cohomology algebra of line bundles on $\P^n$, and
\[
A=\Big(\bigoplus_{p\geq 0,\, q\in\Z} C^p (\P^n,\OO(q)),d\Big)
\]
the \v{C}ech complex with respect to the standard open covering $U_i=(x_i\neq 0)$ of $\P^n$.
There is a~natural dg-algebra structure on $A$, such that the corresponding cohomology algebra is
$H$. The multiplication on $A$ is defined as follows. For $\alpha \in C^p(\P^n,\OO(q))$ and $\beta\in C^{p'}(\P^n,\OO(q'))$, we define $\alpha\beta\in C^{p+p'}(\P^n,\OO(q+q'))$ by
\[
(\alpha\beta)_{i_0i_1\dots i_{p+p'}}:=\alpha_{i_0\dots i_p}|_{U_{i_0\dots i_{p+p'}}}\cdot \beta_{i_p\dots i_{p+p'}}|_{U_{i_0\dots i_{p+p'}}},
\]
where on the right hand side we use the multiplication map $\OO(q)\otimes \OO(q')\to \OO(q+q')$.

The homological perturbation lemma equips $H$ with a minimal $A_\infty$-structure $(m_n)$, where $m_2$ is the
usual product on $H$. We will use the form of this lemma due to Kontsevich--Soibelman~\cite{KS}, which gives formulas
for $m_n$ as sums over trees. To apply homological perturbation, we need the following data:
\begin{itemize}\itemsep=0pt
\item a projection $\pi\colon  A\to H$,
\item an inclusion $\iota\colon  H\to A$, and
\item a homotopy $Q$ such that $\pi\iota=\id_H$ and $\id_A-\iota\pi={\rm d}Q+Q{\rm d}$.
\end{itemize}
Recall that
 $H^0=\C[x_0,\dots ,x_n]$,
 \[
 H^n\simeq \bigoplus_{e_0,\dots ,e_n<0} \k\cdot x_0^{e_0}x_1^{e_1}\cdots x_n^{e_n}\sub A^n,
 \]
 and $H^i=0$ for $i\neq 0,n$. We define $\iota$ in degree zero by $\iota(f)_k=f$
 for $k=0,1,\dots ,n$. We define $\iota$ in degree $n$ by $\iota(g)_{0\ldots n}=g$. We define the projection
 in degree zero to be
\[
\pi(\gamma)=\begin{cases}
\gamma_n & \text{if } \, \gamma_n\in \C[x_0,\dots ,x_n],\\
0 & \text{else}.	
\end{cases}
\]
To define $\pi$ in degree $n$, we observe that
\[A^n=\bigoplus_{e_0,\dots ,e_n\in \Z}\k\cdot x_0^{e_0}x_1^{e_1}\cdots x_n^{e_n},\]
 and we let $\pi$ be the natural projection to $H^n$.

To define the homotopy, we use that $A$ decomposes as a direct sum of chain complexes
\[
A=\oplus_{\vec{e}\in \Z^{n+1}}A(\vec{e}),
\]
where $A(\vec{e})$ consists of all elements
 in $A$ whose components are scalar multiples of $x^{\vec{e}}:=x_0^{e_0}x_1^{e_1}\cdots x_n^{e_n}$. In other words,
 $A(\vec{e})$ is the $\vec{e}$-isotypical summand with respect to the action of the group $\G_m^{n+1}$.

 Let us set for $\vec{e}\in\Z^{n+1}$,
 \[k(\vec{e}):=\max\{i \ |\ e_i\geq 0\}\]
(which is equal to $-\infty$ if all $e_i$ are negative).
 There is then a standard homotopy $Q$ defined on an element $\gamma\in A(\vec{e})^{p}$ by
$Q(\gamma)_{i_0i_1\dots  i_{p-1}}=\gamma_{k(\vec{e})i_0\dots i_{p-1}}$ if $k(\vec{e})>-\infty$
 and $Q(\gamma)_{i_0i_1\dots  i_{p-1}}=0$ otherwise (i.e., if all $e_i$ are negative).

 For a Laurent monomial $x^{\vec{e}}$ and a subset
 $I=\{i_0,\dots ,i_p\}\subset \{0,1,\dots ,n\}$ such that $I\supset \{0\leq i\leq n|e_i<0\}$, let us denote
 by $x^{\vec{e}}_I$ the element of $A^p$ given by
 \[
 (x^{\vec{e}}_I)_{j_0\ldots j_p}=\begin{cases}
	x^{\vec{e}} & \text{if } \{j_0,\dots ,j_p\}=I,\\ 0 &\text{otherwise}.\end{cases}
\]
	Note that the condition $I\supset \{0\leq i\leq n \, |\, e_i<0\}$ guarantees that $x^{\vec{e}}$ is a regular section of the appropriate line bundle over $U_{i_0\ldots i_p}$.
	Clearly, these elements form a basis for $A$ and our homotopy operator $Q$ is given by
	\[
Q\big(x^{\vec{e}}_I\big)=\begin{cases}
		(-1)^jx^{\vec{e}}_{I\setminus k(\vec{e})} & \text{if }k(\vec{e})=i_j\in I,\\
		0 & \text{otherwise}.
	\end{cases}
\]

With these data one can in principle calculate all the higher products on the cohomology algebra $H$. Below, we will get explicit formulas in the case we need.

\subsection[Calculation of m\_4 for P\^{}2]{Calculation of $\boldsymbol{m_4}$ for $\boldsymbol{\P^2}$}\label{m4-P2-calc-sec}

We now specialize to the case of the projective plane $\mathbb{P}^2$. We have no higher products of odd degree because $H$ and $H^{\otimes n}$ only live in even degrees.
Also, for degree reasons the product $m_4$ will only be non-zero on elements $e\otimes f\otimes g\otimes h\in H^{\otimes 4}$ where one or two of the arguments lie in $H^2$ and the rest in $H^0$. Below, we will explicitly compute the product $m_4$ involving one argument in $H^2$.
Thus, the following special case of the multiplication in $A$ will be relevant: for a monomial $x^{\vec{e}}$ and a Laurent monomial $x^{\vec{e}\,'}$, we have
\[
\iota_0\big(x^{\vec{e}}\big)\cdot x^{\vec{e}\,'}_I=x^{\vec{e}\,'}_I\cdot \iota_0\big(x^{\vec{e}}\big)=x^{\vec{e}+\vec{e}\,'}_I.
\]

We use the formula
\[
m_4(e,f,g,h)=-\sum_{T}\epsilon(T)m_T(e,f,g,h),
\]
where the sum runs over all rooted binary trees with 4 leaves labeled $e$, $f$, $g$ and $h$ (from left to right). For each such tree $T$ the expression $m_T(e,f,g,h)$ is computed by moving the inputs through that tree, applying $\iota$ at the leaves, applying the homotopy $Q$ on each interior edge, multiplying elements of $A$ at each inner vertex and finally applying the projection $\pi$ at the bottom.

We have to sum over the following five trees, which we denote $T_1,\dots ,T_5$, respectively,

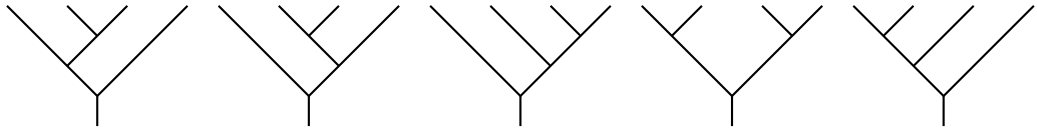
\begin{center}\begin{tikzpicture}[thick, scale=0.4]
\draw (1,5) -- (4,2);
\draw (5,5) -- (3,3);
\draw (3,5) -- (4,4);
\draw (7,5) -- (4,2);
\draw (4,2) -- (4,1);
\end{tikzpicture}\quad\begin{tikzpicture}[thick, scale=0.4]
\draw (1,5) -- (4,2);
\draw (5,5) -- (4,4);
\draw (3,5) -- (5,3);
\draw (7,5) -- (4,2);
\draw (4,2) -- (4,1);
\end{tikzpicture}\quad\begin{tikzpicture}[thick, scale=0.4]
\draw (1,5) -- (4,2);
\draw (5,5) -- (6,4);
\draw (3,5) -- (5,3);
\draw (7,5) -- (4,2);
\draw (4,2) -- (4,1);
\end{tikzpicture}\quad\begin{tikzpicture}[thick, scale=0.4]
\draw (1,5) -- (4,2);
\draw (5,5) -- (6,4);
\draw (3,5) -- (2,4);
\draw (7,5) -- (4,2);
\draw (4,2) -- (4,1);
\end{tikzpicture}\quad\begin{tikzpicture}[thick, scale=0.4]
\draw (1,5) -- (4,2);
\draw (5,5) -- (3,3);
\draw (3,5) -- (2,4);
\draw (7,5) -- (4,2);
\draw (4,2) -- (4,1);
\end{tikzpicture}
\end{center}
Let us first consider the case $e\in H^2$ and $f,g,h\in H^0$ and let's take them all to be basis elements of $H^2$ and $H^0$:
\[
e=\big(x_0^{\alpha_0} x_1^{\alpha_1} x_2^{\alpha_2}\big)_{\{0,1,2\}}, \qquad f=x_0^{a_0}x_1^{a_1}x_2^{a_2}, \qquad g=x_0^{b_0}x_1^{b_1}x_2^{b_2}, \qquad h=x_0^{c_0}x_1^{c_1}x_2^{c_2},
\]
where $\alpha_0,\alpha_1,\alpha_2<0$ and $a_i,b_i,c_i\geq 0$ for $i=0,1,2$. In this case only one of the trees above can be non-zero in the expression for $m_4(e,f,g,h)$, namely $T_5$, because in all other trees at some point the homotopy $Q$ will be applied to an element of $A^0$. Below is a picture of the different summands in $A^\bullet$ and the possible ways the homotopy $Q$ can map a monomial element in each summand:
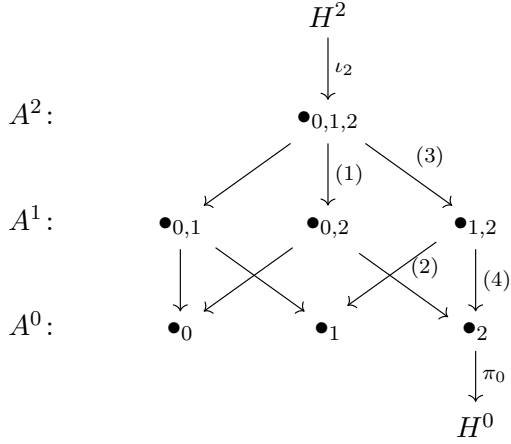
\[\begin{tikzcd}
&&H^2\arrow{d}{\iota_2}&\\
	A^2\colon &&\bullet_{0,1,2}\arrow{dl}{}\arrow{d}{(1)}\arrow{dr}{(3)}&\\
	A^1\colon &\bullet_{0,1}\arrow{d}{}\arrow{dr}{}&\bullet_{0,2}\arrow{dl}{}\arrow[shift right]{dr}{(2)}&\bullet_{1,2}\arrow[shift right]{dl}{}\arrow{d}{(4)}\\
	A^0\colon &\bullet_0&\bullet_1&\bullet_2\arrow{d}{\pi_0}\\
	&&&H^0
\end{tikzcd}
\]
When computing $m_{T_5}(e,f,g,h)$ we should move $e$ through this diagram; at every node it gets multiplied by one of the other arguments and then it moves downwards along one of the arrows. We see that we get non-zero result if we move either along $(1)$ followed by $(2)$ or along $(3)$ followed by $(4)$ (so that we land in $\bullet_2$).
We claim that only the second route is possible. The reason is that at each node we multiply $e$ by a monomial, so the exponents of $x_0$, $x_1$, $x_2$ will not decrease at any time. By the definition of $Q$, if $e$ gets moved along $(1)$ then after the multiplication at $\bullet_{0,1,2}$
the exponent of $x_1$ is non-negative while the exponent of $x_2$ is negative. Hence, after performing the multiplication at $\bullet_{0,2}$ the exponent of $x_1$ is still non-negative.
It follows then from the definition of $Q$ that $e$ cannot move along $(2)$ after moving along $(1)$.

Now comes the computation of $m_{T_5}(e,f,g,h)$. Below, we denote by $\mu$ the multiplication in~$A$. Then
\begin{gather*}
m_{T_5}(e,f,g,h)= \pi\mu(Q\mu(Q\mu(e,f),g),h) \\
\hphantom{m_{T_5}(e,f,g,h)}{}= \pi\mu\big(Q\mu\big(Q\big(x_0^{\alpha_0+a_0}x_1^{\alpha_1+a_1}x_2^{\alpha_2+a_2}\big)_{\{0,1,2\}},g\big),h\big) \\
\hphantom{m_{T_5}(e,f,g,h)}{}\overset{(*)}{=}\pi\mu\big(Q\mu\big(\big(x_0^{\alpha_0+a_0}x_1^{\alpha_1+a_1}x_2^{\alpha_2+a_2}\big)_{\{1,2\}},g\big),h\big)   \\
\hphantom{m_{T_5}(e,f,g,h)}{}=\pi\mu\big(Q\big(x_0^{\alpha_0+a_0+b_0}x_1^{\alpha_1+a_1+b_1}x_2^{\alpha_2+a_2+b_2}\big)_{\{1,2\}},h\big) \\
\hphantom{m_{T_5}(e,f,g,h)}{}\overset{({*}{*})}{=}\pi\big(\mu\big(\big(x_0^{\alpha_0+a_0+b_0}x_1^{\alpha_1+a_1+b_1}x_2^{\alpha_2+a_2+b_2}\big)_{\{2\}},h\big)\big)   \\
\hphantom{m_{T_5}(e,f,g,h)}{}=\pi\big(\big(x_0^{\alpha_0+a_0+b_0+c_0}x_1^{\alpha_1+a_1+b_1+c_1}x_2^{\alpha_2+a_2+b_2+c_2}\big)_{\{2\}}\big) \\
\hphantom{m_{T_5}(e,f,g,h)}{}\overset{({*}{*}{*})}{=} x_0^{\alpha_0+a_0+b_0+c_0}x_1^{\alpha_1+a_1+b_1+c_1}x_2^{\alpha_2+a_2+b_2+c_2},
\end{gather*}
where the symbols $(*)$, $({*}{*})$ are $({*}{*}{*})$ mean that we get zero unless the following conditions hold:
\begin{gather*}
	(*)\ \begin{cases}
	\alpha_0+a_0\geq 0,\\
	\alpha_1+a_1<0,\\
	\alpha_2+a_2<0,
			\end{cases} \quad
	({*}{*})\ \begin{cases}
	\alpha_1+a_1+b_1\geq 0,\\
	\alpha_2+a_2+b_2<0,
			\end{cases}\quad
	({*}{*}{*})\ \begin{cases}
	\alpha_0+a_0+b_0+c_0\geq0,\\
	\alpha_1+a_1+b_1+c_1\geq0,\\
	\alpha_2+a_2+b_2+c_2\geq 0.
	\end{cases}
\end{gather*}
In the end, we have
\[
m_4(e,f,g,h)=-m_{T_5}(e,f,g,h)=
-\rho\big(\vec{\a};\vec{a},\vec{b},\vec{c}\big)\cdot x^{\vec{\a}+\vec{a}+\vec{b}+\vec{c}},
\]
where
\[
\rho\big(\vec{\a};\vec{a},\vec{b},\vec{c}\big):=\begin{cases}
  1 & \text{if} \ \alpha_0+a_0\geq 0,
    \ \alpha_1+a_1<0,
	\ \alpha_1+a_1+b_1\geq 0,\\
    &\hphantom{\text{if}}\ \alpha_2+a_2+b_2<0,
	\ \alpha_2+a_2+b_2+c_2\geq 0,\\
	0 & \text{else}.
\end{cases}\]

Similarly, we compute $m_4$ applied to $e$, $f$, $g$, $h$ in any given order.
We have
\begin{gather*}
m_4(e,f,g,h)=
-\rho\big(\vec{\a};\vec{a},\vec{b},\vec{c}\big)\cdot x^{\vec{\a}+\vec{a}+\vec{b}+\vec{c}},\\
m_4(f,e,g,h)=\big[-\rho\big(\vec{\a};\vec{a},\vec{b},\vec{c}\big)+\rho\big(\vec{\a};\vec{b},\vec{a},\vec{c}\big)-
\rho\big(\vec{\a};\vec{b},\vec{c},\vec{a}\big)\big]\cdot x^{\vec{\a}+\vec{a}+\vec{b}+\vec{c}},\\
m_4(f,g,e,h)=
\big[\rho\big(\vec{\a};\vec{b},\vec{a},\vec{c}\big)-
\rho\big(\vec{\a};\vec{b},\vec{c},\vec{a}\big)+\rho\big(\vec{\a};\vec{c},\vec{b},\vec{a}\big)\big]\cdot x^{\vec{\a}+\vec{a}+\vec{b}+\vec{c}},\\
m_4(f,g,h,e)=
\rho\big(\vec{\a};\vec{c},\vec{b},\vec{a}\big)\cdot x^{\vec{\a}+\vec{a}+\vec{b}+\vec{c}}.
\end{gather*}

\section{Feigin--Odesskii brackets}

\subsection{Bivectors on projective spaces}

It is well known that every $\G_m$-invariant bivector on a vector space $V$ leads to a bivector on the projective space $\P V$.
A bivector on $V$ can be thought of as a skew-symmetric bracket $\{\cdot,\cdot\}$ on the polynomial algebra $S(V^*)$,
which is a biderivation. Such a bracket is $\G_m$-invariant if and only if the bracket of two linear forms is a quadratic form.
In other words, such a bracket can be viewed as a skew-symmetric pairing
\[
b\colon \ V^*\times V^*\to S^2(V^*).
\]
The corresponding bivector $\Pi$ on the projective space $\P V$ is determined by the skew-symmetric forms
$\Pi_v$ on $T_v^* \P V$ for each point $\lan v\ran \in \P V$. We have an identification
\[
T_v^* \P V=\lan v\ran^\vee \sub V^*.
\]
It is easy to see that under this identification we have
\begin{equation}\label{Pi-b-formula}
\Pi_v(s_1\we s_2)=b(s_1,s_2)(v),
\end{equation}
where $s_1,s_2\in \lan v\ran^\vee$. Here we take the value of the quadratic form $b(s_1\we s_2)$ at $v$.

We can use the above formula in reverse. Namely, suppose for some bivector $\Pi$ on $\P V$ we found a skew-symmetric pairing $b$ such that
\eqref{Pi-b-formula} holds. Then the $\G_m$-invariant bracket $\{\cdot,\cdot\}$ on $S(V)$ given by $b$ induces the bivector $\Pi$ on $\P V$.
Note that if $\Pi$ is a Poisson bivector on $\P V$, it is not guaranteed that the $\G_m$-invariant bracket $\{\cdot,\cdot\}$ on $S(V)$ is also Poisson, i.e.,
satisfies the Jacobi identity (but it is known that $\{\cdot,\cdot\}$ can be chosen to be Poisson, see~\cite{B, P95}).

\subsection[Recollections from~\protect{[3]}]{Recollections from~\cite{HP}}

Below, we will denote simply by $L_1L_2$ the tensor product of line bundles $L_1$ and $L_2$.

Let $\xi$ be a line bundle of degree $n$ on an elliptic curve $C$. We fix a trivialization $\om_C\simeq \OO_C$.
Then the associated Feigin--Odesskii Poisson structure $\Pi$ (to which we will refer as {\it FO bracket}) on $\P H^1\big(\xi^{-1}\big)\simeq \P H^0(\xi)^*$ is given by the formula (see~\cite[Lemma~2.1]{HP})
\begin{equation}\label{Pi-m3-formula-eq}
\Pi_{\phi}(s_1\we s_2)=\lan \phi, {\rm MP}(s_1,\phi,s_2)\ran,
\end{equation}
where $\lan\phi\ran\in \P\Ext^1(\xi,\OO)$, and $s_1,s_2\in \lan \phi\ran^\perp$.
Here we use the Serre duality pairing $\lan \cdot,\cdot\ran$ between~$H^0(\xi)$ and $H^1\big(\xi^{-1}\big)$ and the triple Massey product
\[
{\rm MP}\colon \ H^0(\xi)\ot H^1\big(\xi^{-1}\big)\ot H^0(\xi)\to H^0(\xi)
\]
that also agrees with the triple product $m_3$ obtained by homological perturbation from the natural dg enhancement of the derived category of coherent sheaves on $C$.
There is some ambiguity in a choice of $m_3$ but for $s_1,s_2\in \lan\phi\ran^{\perp}$, the expression in the right-hand side of
\eqref{Pi-m3-formula-eq} is well defined.

Next, assume that $S$ is a smooth projective surface, $L$ is a line bundle on $S$ such that $H^*(S,LK_S)=0$, and let $C\sub S$
be a smooth connected anticanonical divisor (which is an elliptic curve), so we have an exact sequence of coherent sheaves on $S$,
\begin{equation}
0\to K_S\overset{F}{\to} \OO_S\to \OO_C\to 0.\label{KS-O-OC-ex-seq}
\end{equation}
We have a natural restriction map
\[H^0(S,L)\to H^0(C,L|_C).\]
The exact sequence
\begin{equation}
0\to LK_S\overset{F}{\to} L\to L_C\to 0\label{LKS-L-LC-ex-seq}
\end{equation}
shows that under our assumptions this restriction map is an isomorphism.

Thus, the FO bracket on $\P H^0(L|_C)^*$ associated with $(C,L|_C)$ (defined up to rescaling) can be viewed as a Poisson structure on
a fixed projective space $\P V^*$, where
\[V:=H^0(S,L).\]
By~\cite[Theorem~4.4]{HP}, the Poisson brackets on $\P V^*$ associated with different anticanonical divisors are compatible.
More precisely, we get a linear map from $H^0\big(S,K_S^{-1}\big)$ to the space of bivectors on $\P V^*$, whose image lies in the space of Poisson brackets.

\subsection{Feigin--Odesskii bracket for an anticanonical divisor}

We keep the data $(S,L)$ of the previous subsection. Let $i\colon C\hra S$ be an anticanonical divisor in~$S$, with the equation $F\in H^0\big(S,K_S^{-1}\big)$.
We want to write a formula for the FO bracket $\Pi=\Pi_F$
on $\P V^*$ in terms of higher products on the surface $S$ and the equation $F$. For this we rewrite the right-hand side of formula~\eqref{Pi-m3-formula-eq}.
Let us write the triple product in this formula as~${\rm MP}^C$ to remember that it is defined for the derived category of $C$.

\begin{Proposition}\label{m4-prop}\quad
\begin{enumerate}\itemsep=0pt
\item[$(i)$] In the above situation,
given $e\in V^*$ and $s_1,s_2\in \lan e\ran^{\perp}$, one has
\[
\big\lan e,{\rm MP}^C(s_1|_C,e,s_2|_C)\big\ran=\lan m_4(F,s_1,e,s_2)-m_4(s_1,F,e,s_2),e\ran,\]
where we use the identification $V^*\simeq H^2\big(S,L^{-1}K_S\big)$ given by Serre duality and consider the $A_\infty$-products on $S$,
\begin{gather*}
m_4\colon \ H^0\big(K_S^{-1}\big)H^0(L)H^2\big(L^{-1}K_S\big)H^0(L)\to H^0(L), \\
\hphantom{m_4\colon}{} \ H^0(L)H^0\big(K_S^{-1}\big)H^2\big(L^{-1}\big)H^0(L)\to H^0(L),
\end{gather*}
obtained by the homological perturbation.

\item[$(ii)$] Assume that a generic anticanonical divisor is smooth $($and connected$)$. Then
\[\Pi_F|_e(s_1\we s_2):=\lan m_4(F,s_1,e,s_2)-m_4(s_1,F,e,s_2),e\ran\]
gives a collection of compatible Poisson brackets on $\P V$
depending linearly on $F$.
\end{enumerate}
\end{Proposition}

\begin{proof}
(i) By Serre duality, $H^*\big(S,L^{-1}\big)=0$, so
the map
\[
H^1\big(C,L^{-1}\big|_C\big)\to H^2\big(S,L^{-1}K_S\big),
\]
induced by the exact sequence
\[
0\to L^{-1}K_S\to L^{-1}\to L^{-1}\big|_C\to 0,
\]
is an isomorphism.
It is a standard fact that this isomorphism is the dual to the isomorphism $H^0(S,L)\to H^0(C,L|_C)$ given by the restriction, via Serre dualities on $S$ and $C$.
Let us denote by $e_C\in H^1\big(C,L^{-1}\big|_C\big)$ the element corresponding to $e\in H^2\big(S,L^{-1}K_S\big)$ under the above isomorphism.

We claim that the triple Massey product ${\rm MP}^C(s_1|_C,e_C,s_2|_C)=m_3(s_1|_C,e_C,s_2|_C)$ corresponding to the arrows
\[\OO_C\overset{s_2|_C}{\longrightarrow} L|_C \overset{e_C}{\longrightarrow} \OO_C\overset{s_1|_C}{\longrightarrow} L|_C\]
(where the middle arrow has degree $1$) agrees with the corresponding triple Massey product on~$S$,
\begin{gather*}
\OO_S\overset{s_2}{\longrightarrow} L\overset{e_C}{\longrightarrow}\OO_C\overset{s_1|_C}{\longrightarrow} L|_C.
\end{gather*}
Indeed, the relevant spaces are identified via the restriction maps. Let
\[
r\colon \ \OO_S\to \OO_C, \qquad r_L\colon \ L\to L|_C
\]
 be the natural maps.
Then we have to check that for $s_1,s_2\in \lan e\ran^\perp\sub H^0(S,L)$, one has
\[m_3(s_1|_C,e_C,s_2|_C) r\equiv m_3(s_1|_C,e_Cr_L,s_2) \mod \lan s_1|_C r,s_2|_C r\ran,\]
where we view this as equality of cosets in $\Hom(\OO_S,L|_C)$.
The $A_\infty$-identities imply that
\[m_3(s_1|_C,e_C,s_2|_C)r= m_3(s_1|_C,e_C,s_2|_Cr)\pm s_1|_Cm_3(e_C,s_2|_C,r),\]
where $s_2|_Cr=r_Ls_2$, and
\[m_3(s_1|_C,e_C,r_Ls_2)=m_3(s_1|_C,e_Cr_L,s_2)\pm s_1|_Cm_3(e_C,r_L,s_2)\pm m_3(s_1|_C,e_C,r_L)s_2.\]
Combining these two identities, we deduce our claim.

Thus, it is enough to calculate the Massey product ${\rm MP}(s_1|_C,e_Cr_L,s_2)$.
Using the exact sequences~\eqref{KS-O-OC-ex-seq} and~\eqref{LKS-L-LC-ex-seq}, we can represent $\OO_C$ (resp.\ $L_C$) by the twisted complex $[K_S[1]\to \OO_S]$
(resp.\ $[LK_S[1]\to L]$).

In terms of these resolutions, the elements of $\Ext^1(L,\OO_C)$ get represented by $\Ext^2(L,K_S)\sub \hom^\bullet(L,[K_S[1]\to \OO_S])$,
while the element of $\Hom(\OO_C,L|_C)$ corresponding to $s\in H^0(S,L)\simeq H^0(C,L|_C)$ is given by the natural map of twisted complexes
induced by the multiplication by~$s$. The elements of $\Hom(\OO_S,L|_C)$ are identified with $\Hom(\OO_S,L)\simeq \hom^0(\OO_S,[LK_S[1]\to L])$.
Thus, the $m_3$ product we are interested is given by the following triple product in the category of twisted complexes over $S$:
\[
\begin{tikzcd}[row sep=3em,column sep=3em]
\OO_S \arrow[d,"s_2"'] & \\
L \arrow[d,"e"'] & \\
K_S[1] \arrow[r,"F"]\arrow[d,"s_1"'] & \OO_S \arrow[d,"s_1"']\\
LK_S[1]\arrow[r,"F"] & \hphantom{,}L,
\end{tikzcd}
\]
where we view $e$ as a morphism of degree $1$ from $L$ to $K_S[1]$.
Now the formula for $m_3$ on twisted complexes (see~\cite[Section~7.6]{Keller}) gives
\[m_4(F,s_1,e,s_2)-m_4(s_1,F,e,s_2)\]
(here the insertions of $F$ correspond to insertions of the differentials in the twisted complexes).

(ii) It is clear that $\Pi_F$ gives a linear map from $H^0\big(S,\om_S^{-1}\big)$ to the space of bivectors on $\P V$. By~(i), for generic $F$ we get a Poisson bracket.
Hence, this is true for all $F$.
\end{proof}

\subsection[The case leading to 10 compatible brackets on P\^{}5]{The case leading to $\boldsymbol{10}$ compatible brackets on $\boldsymbol{\P^5}$}

We can apply Proposition \ref{m4-prop} to the case $S=\P^2$ and $L=\OO(2)$. Note that the assumptions are satisfied in this case since
$LK_S=\OO(-1)$ has vanishing cohomology. Thus, for each $F\in H^0\big(\P^2,\OO(3)\big)$ giving a smooth cubic, we get a formula for the FO-bracket
$\Pi_F$ on $\P H^0\big(\P^2,\OO(2)\big)^*=\P^5$.
Hence, we get a family of $10$ (the dimension of $H^0\big(\P^2,\OO(3)\big)$
compatible brackets on $\P^5$ (we also know this from~\cite[Proposition~4.7]{HP}). The fact that these $10$ brackets are linearly independent follows from the compatibility of this
construction with the ${\rm GL}_3$-action and is explained in~\cite[Proposition~4.7]{HP}.

Now we will derive formulas for the brackets $\{\,,\,\}_F$ on the algebra of polynomials in $6$ variables which induce the above Poisson brackets on $\P V\simeq \P^5$,
where
\[
V=H^0\big(\P^2,\OO(2)\big)^*.
\]
They depend linearly on $F$, so we will just give formulas for $\{\,,\,\}_{x^{\vec{c}}}$, where $x^{\vec{c}}$ runs through all 10 monomials of degree $3$ in $(x_0,x_1,x_2)$.

Let us set
\[\Delta(n):=\begin{cases}
	\big\{(a_0,a_1,a_2)\in\Z^3\mid a_0+a_1+a_2=n,\, a_i\geq 0 \ \text{for} \ i=0,1,2\big\}& \text{if }n\geq 0,\\
	\big\{(\alpha_0,\alpha_1,\alpha_2)\in \Z^3\mid \alpha_0+\alpha_1+\alpha_2=n,\, \alpha_i<0 \ \text{for}  \ i=0,1,2\big\} & \text{if }n<0.
\end{cases}\]
Note that $\{x^{\vec{e}}\mid \vec{e}\in \Delta(n)\}$ forms a basis for $H^0\big(\P^2,\OO(n)\big)$ when $n\geq 0$, while
$\big\{x^{\vec{e}}_{\{0,1,2\}}\mid \vec{e}\in \Delta(n)\big\}$ is a basis for $H^2\big(\P^2,\OO(n)\big)$ when $n<0$.
In particular, we use $\big\{x^{\vec{a}} \mid \vec{a}\in\Delta(2)\big\}$ as a basis in $V^*=H^0\big(\P^2,\OO(2)\big)$. Our brackets should associate to a pair of elements of this basis a quadratic form
in the same variables.

\begin{Theorem}\label{main-thm} One has for $\vec{a},\vec{b}\in\Delta(2)$, $\vec{c}\in\Delta(3)$,
\begin{equation}\label{main-formula-eq}
\big\{x^{\vec{a}},x^{\vec{b}}\big\}_{x^{\vec{c}}}:=\sum_{\vec{a}\,{}',\vec{b}\,{}'\in \Delta(2)}\bigg[\sum_{\sigma}-\operatorname{sgn}(\sigma)\tilde{\rho}\big(\sigma\vec{a}, \sigma\vec{b},\sigma\vec{c},\vec{a}\,{}',\vec{b}\,{}'\big)\bigg]x^{\vec{a}\,{}'}x^{\vec{b}\,{}'},
\end{equation}
	where the second sum is over the symmetric group on the letters $\{a,b,c\}$ and
	\[\tilde{\rho}\big(\vec{a},\vec{b},\vec{c},\vec{a}\,{}',\vec{b}\,{}'\big):=\begin{cases}
  1 & \text{if} \ a_0'\leq a_0-1, \ a_1'>a_1-1,\ a_1'\leq a_1+b_1-1,\\
	
    & \hphantom{\text{if}}{}\ a_2+b_2<a_2'+1,\ c_2+a_2+b_2\geq a_2'+1,\\
	
	& \hphantom{\text{if}}{}\ a_0'+b_0'=a_0+b_0+c_0-1,\ a_1'+b_1'=a_1+b_1+c_1-1,\\
	0 & \text{else}.
\end{cases}\]
\end{Theorem}
\begin{proof}
By Serre duality, we can identify $V=H^0\big(\P^2,\OO(2)\big)^*$ with $H^2\big(\P^2,\OO(-5)\big)$.
By Proposition \ref{m4-prop}, the bracket $\{x^{\vec{a}},x^{\vec{b}}\}_{x^{\vec{c}}}$ is the quadratic form on $V\simeq H^2\big(\P^2,\OO(-5)\big)$ given by
\[
Q(e):=\big\langle e,m_4\big(x^{\vec{c}},x^{\vec{a}},e,x^{\vec{b}}\big)-m_4\big(x^{\vec{a}},x^{\vec{c}},e,x^{\vec{b}}\big)\big\rangle.
\]
We can write
\[
e=\sum_{\vec{\alpha}\in \Delta(-5)}c_{\vec{\alpha}}x^{\vec{\alpha}}_{\{0,1,2\}}\in H^2\big(\P^2,\OO(-5)\big).
\]
Using the formulas for $m_4$ from the end of Section \ref{m4-P2-calc-sec},
we get
	\[
Q(e)=\sum_{\vec{\alpha},\vec{\beta}\in \Delta(-5)}\bigg[\sum_{\sigma}-\operatorname{sgn}(\sigma)\rho\big(\vec{\alpha};\sigma\vec{a},\sigma\vec{b},\sigma\vec{c}\big)\bigg]\delta\big(\vec{\alpha},\vec{\beta},\vec{a},\vec{b},\vec{c}\big)c_{\vec{\alpha}}c_{\vec{\beta}},
\]
	where the second sum runs over the symmetric group on the letters $\{a,b,c\}$ and
	\[
\delta(\vec{\alpha},\vec{\beta},\vec{a},\vec{b},\vec{c})=\begin{cases}
		1 & \text{if }\vec{\alpha}+\vec{\beta}+\vec{a}+\vec{b}+\vec{c}=(-1,-1,-1),\\
		0 & \text{else}.
	\end{cases}
\]

 We have to show that the element in $S^2\big(H^0\big(\P^2,\OO(2)\big)\big)$ given by the right-hand side of~\eqref{main-formula-eq}
defines the same quadratic form $Q$ on $H^2\big(\P^2,\OO(-5)\big)$.
To see this, we apply it to the element $e=\sum_{\vec{\alpha}\in \Delta(-5)}c_{\vec{\alpha}}x^{\vec{\alpha}}_{\{0,1,2\}}\in H^2\big(\P^2,\OO(-5)\big)$. For $\vec{\alpha}\in \Delta(-5)$, we set $\vec{\alpha}^*:=(-1,-1,-1)-\vec{\alpha}$ and then we compute
\begin{gather*}
	\bigg(\sum_{\vec{a}\,{}',\vec{b}\,{}'\in \Delta(2)}\bigg[\sum_{\sigma}-\operatorname{sgn}(\sigma)\tilde{\rho}\big(\sigma\vec{a}, \sigma\vec{b},\sigma\vec{c},\vec{a}\,{}',\vec{b}\,{}'\big)\bigg]x^{\vec{a}\,{}'}x^{\vec{b}\,{}'}\bigg)(e)\\
\qquad{}=	\sum_{\vec{\alpha},\vec{\beta}\in \Delta(-5)}\sum_{\vec{a}\,{}',\vec{b}\,{}'\in \Delta(2)}\!\!\bigg[\sum_{\sigma}-\operatorname{sgn}(\sigma)\tilde{\rho}\big(\sigma\vec{a}, \sigma\vec{b},\sigma\vec{c},\vec{a}\,{}',\vec{b}\,{}'\big)\bigg]\langle x^{\vec{a}\,{}'},x^{\vec{\alpha}}_{\{0,1,2\}}\rangle \langle x^{\vec{b}\,{}'},x^{\vec{\beta}}_{\{0,1,2\}}\rangle c_{\vec{\alpha}}c_{\vec{\beta}}\\
\qquad{}=	\sum_{\vec{\alpha},\vec{\beta}\in \Delta(-5)}\bigg[\sum_{\sigma}-\operatorname{sgn}(\sigma)\tilde{\rho}\big(\sigma \vec{a},\sigma \vec{b},\sigma \vec{c},\vec{\alpha}^*,\vec{\beta}^*\big)\bigg]c_{\vec{\alpha}}c_{\vec{\beta}}.
\end{gather*}
Now it only remains to note that for any permutation $\sigma$, one has
\[
\tilde{\rho}\big(\sigma\vec{a}, \sigma\vec{b},\sigma\vec{c},\vec{\alpha}^*,\vec{\beta}^*\big)=\rho\big(\vec{\alpha};\sigma\vec{a}, \sigma\vec{b},\sigma\vec{c}\big)\delta\big(\vec{\alpha},\vec{\beta},\vec{a},\vec{b},\vec{c}\big),
\]
for $\tilde{\rho}$ given in the formulation of the theorem. \end{proof}

\begin{Remarks}\quad
\begin{enumerate}\itemsep=0pt
\item Note that when we take $\vec{c}=(0,0,3)$ only two permutations $\sigma$, namely, $\sigma=1$ and $\sigma=(a \ b)$, can give non-zero terms in the formula of
Theorem \ref{main-thm}. When $\vec{c}=(1,2,0)$ all permutations except $\sigma=1$ and $\sigma=(a \ b)$ may give non-zero terms. When $\vec{c}=(1,1,1)$ all permutations can give non-zero terms.

\item It is not true that formulas~\eqref{main-formula-eq} define compatible Poisson brackets on the algebra of polynomials in $6$ variables: this is true only for the induced
brackets on $\P^5$ (in other words, the relevant identities hold only for the ratios of coordinates $x_i/x_j$).
\end{enumerate}
\end{Remarks}

\subsection*{Acknowledgements}

We thank the anonymous referee for useful remarks.
A.P. is partially supported by the NSF grant DMS-2001224,
and within the framework of the HSE University Basic Research Program and by the Russian Academic Excellence Project `5-100'.

\pdfbookmark[1]{References}{ref}
\LastPageEnding

\end{document}